\documentclass[a4paper,11 pt]{amsart}
\usepackage{hyperref}

\usepackage[usenames,dvipsnames,svgnames,table]{xcolor}
\usepackage[all]{xy}
 \usepackage{enumitem} 
 \usepackage[normalem]{ulem}
 \usepackage{multicol}
\usepackage{longtable}

\newcommand{\tsk}[1]{\textcolor{YellowOrange}}
\newcommand{\veq}{\mathrel{\rotatebox{90}{$=$}}}

\usepackage{tikz}
\usepackage{tikz-cd}

\makeatletter
\def\@endtheorem{\endtrivlist}
\makeatother

\usepackage[active]{srcltx}
\usepackage{amsmath}
\usepackage{amssymb}
\usepackage{amscd}
\usepackage{amsthm}
\usepackage[latin1]{inputenc}
\usepackage{mathrsfs}  

\usepackage{nicefrac}

\newtheorem{teo}{Theorem}[section]
\newtheorem{defn}[teo]{Definition}
\newtheorem{prop}[teo]{Proposition}

\theoremstyle{definition}

\newtheorem{remark}[teo]{Remark}

\newtheorem*{claim}{Claim}

\newtheoremstyle{dico}
 {\baselineskip}   
  {\topsep}   
  {}  
  {0pt}       
  {} 
  {.}         
  {5pt plus 1pt minus 1pt} 
  {}          
\theoremstyle{dico}

\numberwithin{equation}{section}

\newsavebox{\overlongequation}

\newcommand{\ra}{\rightarrow}
\newcommand{\C}{\mathbb{C}}

\newcommand{\Zeta}{{\mathbb{Z}}}
\newcommand{\Z}{{\mathbb{Z}}}

\newcommand{\vacuo}{\emptyset}

\newcommand{\restr}[1]          {\vert_{#1}}

\renewcommand{\setminus}{-}

\renewcommand{\phi}{\varphi}
\newcommand{\lds}{\ldots}

\newcommand{\ga}{\gamma}
\newcommand{\Ga}{\Gamma}

\newcommand{\rank}{\operatorname{rank}}

\newcommand{\OO}{\mathcal{O}}

\newcommand{\PP}{\mathbb{P}}

\renewcommand{\phi}             {\varphi}

\newcommand{\mm}{{\mathbf{m}}}
\newcommand{\jac}{\mathsf{T}^0_g}
\newcommand{\tor}{\mathsf{T}_g}

\newcommand{\ag}{\mathsf{A}_g}

\newcommand{\zg}{\mathsf{Z}}
\newcommand{\datum}{{(\mm, G, \theta)}}

\newcommand{\Diff}{\operatorname{Diff}}
\newcommand{\Map}{\operatorname{Map}}

\newcommand{\ut}{U_t}

\begin{document}

\author{Paola Frediani, Paola Porru}

\title{On the Bielliptic and bihyperelliptic loci}

%
%

%
%
%
%

\begin{abstract}
We study some particular loci inside the moduli space $\mathcal{M}_g$, namely the bielliptic locus (i.e. the locus of curves admitting a $2:1$ cover over an elliptic curve $E$) and the bihyperelliptic locus (i.e. the locus of curves admitting a $2:1$ cover over a hyperelliptic curve $C'$, $g(C') \geq 2$). We show that the bielliptic locus is not a totally geodesic subvariety of $\mathcal{A}_g$ if $g \geq 4$ (while it is for $g=3$, see \cite{frediani2016shimura}) and that the bihyperelliptic locus is not totally geodesic in $\mathcal{A}_g$  if $g \geq 3g'$. We also give a lower bound for the rank of the second gaussian map at the generic point of the bielliptic locus and an upper bound for this rank for every bielliptic curve.   
\end{abstract}

 \thanks{The  authors were partially supported by MIUR PRIN 2015``Geometry of Algebraic Varieties''.
   The first author was also partially supported by FIRB 2012 `` Moduli Spaces and their Applications''. The authors were also
   partially supported by GNSAGA of INdAM.  }
\maketitle

\setcounter{tocdepth}{1}
\tableofcontents{}

\section{Introduction}

Denote by $\mathcal{M}_g$ the moduli space parametrizing compact Riemann surfaces of genus $g$, and denote by $\mathcal{A}_g$ the moduli space of principally polarized abelian varieties of dimension $g$. Call $j: \mathcal{M}_g \ra \mathcal{A}_g$ the Torelli map, associating to a smooth complex algebraic curve of genus $g$ its Jacobian variety. Denote by $\jac$ the image of the Torelli map. Call its closure \emph{Torelli locus} and denote it by $\tor$. 

Both $\mathcal{M}_g$ and $\mathcal{A}_g$ are complex orbifolds, and $\mathcal{A}_g$ is endowed with a locally symmetric metric, the so-called Siegel metric, induced by the symmetric metric on the Siegel space $H_g$ of which ${\mathcal A}_g$ is the quotient by the action of $Sp(2g, \Zeta)$.

A conjecture by Coleman and Oort says that for large enough genus, there should not exist any positive-dimensional Shimura (or special) subvariety $\zg$ of $\mathcal{A}_g$ such that $\zg \subset \tor$ and $\zg\cap\jac\neq \vacuo$~\cite{oort1997canonical}.  

Special subvarieties of $\mathcal{A}_g$ are totally geodesic algebraic subvarieties (with respect to the Siegel metric) admitting a CM point (\cite{mumford1969note}, ~\cite{moonen1998linearity}), so one possible approach to this problem is via  the study the second fundamental form of the period map (\cite{ colombo2015totally}, \cite{colombo2010siegel}, \cite{colombo2000hodge}).

For low genus ($g \leq 7$), examples of Shimura  subvarieties contained in the Jacobian locus are known. All of them are constructed as families of Jacobians of Galois coverings of the line (see \cite{shimura1964purely, mostow1988discontinuous, de1991jacobians, rohde2009introduction, moonen2010special,moonen2011torelli,frediani2015shimura, mohajer2017shimura}) and of elliptic curves (see~\cite{frediani2016shimura}, \cite{grushevsky2013shimura}).

All the examples of families of Galois coverings constructed so far yielding a Shimura subvariety of ${\mathcal A}_g$ satisfy a condition that we briefly explain. 
Consider  a Galois cover $f: C \rightarrow C' = C/G$. Call $g=g(C)$ and $g'=g(C')$. Then $G$ injects in the mapping class group $ \Map_g := \pi_0 ( \Diff^+ (C))$. The fixed point locus ${\mathcal T}_g^G$ of the action of $G$ on  the Teichm\"uller space ${\mathcal T}_g$  is a complex submanifold of dimension $3g'-3 + r$, where $r$ denotes the number of critical values of the map $f$.  Denote by $\mathsf{M}$ the image of  ${\mathcal T}_g^G$ in $\mathcal{M}_g$ under the natural map  ${\mathcal T}_g \ra  \mathcal{M}_g$. We study the variety defined by the closure $\mathsf{Z}$ of the image of $\mathsf{M}$ in $\mathcal{A}_g$ via the Torelli map.

Set $N:= \dim (S^2 H^0(C, K_C)) ^ G$. The sufficient condition, which we denote by $(*)$, for a subvariety $\mathsf{Z}$ as above to be special (see \cite{colombo2015totally,frediani2015shimura}) is: $$ (*) \ \ N = 3g'-3+r.$$

 In \cite{colombo2015totally}, it is proven that  condition $(*)$  implies that the subvariety $\mathsf{Z}$ is totally geodesic  and in \cite{frediani2015shimura} it is proven that, in fact, it is a Shimura subvariety in case $g'=0$, and the same proof also works if $g' >0$.

In \cite{colombo2015totally} the authors gave the complete list of all the families of Galois coverings of $\PP^1$ of genus $g \leq 9$ satisfying condition $(*)$ and hence yielding Shimura subvarieties of $\ag$ contained in the Torelli locus. In \cite{frediani2016shimura} the complete list of all families of Galois coverings of an elliptic curve satisfying condition $(*)$ is given for every genus. Moreover, in the same paper, the authors prove that if a family satisfies this sufficient condition with $g'\geq 1$, then $g \leq 6g'+1$. 

One of the goals of this paper is to investigate condition $(*)$: it is a sufficient condition for $\mathsf{Z}$ to be special.  In ~\cite{moonen2010special} Moonen proved  using techniques in arithmetic geometry that condition $(*)$ is also necessary in the case of cyclic covers of $\PP^1$.  
A similar result is proven in \cite{mohajer2017shimura} in the case of abelian covers when the dimension of the family is one.

 For the general case, we don't know if condition $(*)$ is also necessary.  
In \cite[Proposition  5.2]{colombo2015totally} it is proven  that, under some additional assumptions,  the families of cyclic cover of $\PP^1$ not satisfying $(*)$ are  not even totally geodesic.

Totally geodesic submanifolds of $\mathcal{A}_g$ are related to the second Gaussian map of the canonical bundle. In particular in~\cite{colombo2000hodge} it is proven that the second fundamental form of the orbifold immersion $j: \mathcal{M}_g \rightarrow \mathcal{A}_g$ (the immersion holds outside the hyperelliptic locus, see~\cite{oort1979local} for details) lifts the second Gaussian map of the canonical bundle, as stated in an unpublished paper of Green and Griffiths (see \cite{green1994infinitesimal}). In~\cite{colombo2000hodge}, an explicit expression for the second fundamental form when evaluated on Schiffer variations is provided (see also~\cite[Theorem 2.6]{colombo2015totally}). More precisely, $\rho(\xi_p \odot \xi_p)$ reduces, up to a constant, to the evaluation of the second Gaussian map at the point $p$. However, it is much more difficult to use the expression given in~\cite{colombo2000hodge} to compute the second fundamental form on $\xi_p \odot \xi_q$, when $p \neq q$. In this case, in fact, the formula contains the evaluation at $q$ of a meromorphic $1$-form on the curve, called $\eta_p$, which has a double pole at $p$ and is defined by Hodge theory. 

%
The expression of the second fundamental form of the period map given in \cite{colombo2000hodge} has been used in  \cite [Thm. 4.2]{colombo2015totally}  to get  an upper bound for the dimension  of a germ of a totally geodesic submanifold  of ${\mathcal A}_g$ contained in the Jacobian locus  passing through $J(C)$ depending on the gonality of $C$. From this one gets an upper bound  for the dimension  of a germ of a totally geodesic submanifold  of ${\mathcal A}_g$ contained in the Jacobian locus in terms of $g$ (\cite [Thm. 4.4]{colombo2015totally}).


%
Notice that the upper bound of \cite [Thm. 4.2]{colombo2015totally}  in the case of tetragonal curves is $2g$ and the bielliptic locus has dimension $2g-2$. 

In~\cite{frediani2016shimura} it is proven that for $g =3$, the bielliptic locus yields a Shimura subvariety, in particular it  is totally geodesic. We denote by ${\mathcal B}_g$ the bielliptic locus and by $\mathcal{BH}_{g,g'}$  the bihyperelliptic locus, that is the locus of curves of genus $g$ admitting a $2:1$ map to a hyperelliptic curve of genus $g'$.

We prove the following.

\begin{teo}
The bielliptic locus $\mathcal{B}_g$ is not totally geodesic when $g \geq 4$. The bihyperelliptic locus $\mathcal{BH}_{g,g'}$  is not totally geodesic when $g \geq 3g'$.
\end{teo}

In  \cite{ballico2007geometry} it is proven that if $g \geq 4g' +2$ the locus $\mathcal{BH}_{g,g'}$ is irreducible and unirational of dimension $2g-2g'+1$ (note that in  \cite{ballico2007geometry} there is a small inaccuracy in the statement of the main theorem about  the dimension of  $\mathcal{BH}_{g,g'}$). For $3g' \leq g < 4g' +2$ we prove that there exists at least one irreducible component that is not totally geodesic.

We point out that in~\cite{frediani2016shimura} it is proven that  if $g \geq 4$ the bielliptic locus does not satisfy condition (*). In particular our result proves that condition (*)  is necessary for bielliptic curves to yield Shimura subvarieties of ${\mathcal A}_g$.

The proof of the above result is based on the construction of  suitable quadrics $Q$ which are  invariant under the  bielliptic and the bihyperelliptic involution respectively. Then we find a  pair of invariant elements $v_1, v_2 \in H^1(C,T_C)$ such that the second fundamental form $\rho(Q)(v_1 \odot v_2) \neq 0$. This shows that these loci are not totally geodesic.

In the case of bielliptic curves the proof is simplified by  the fact that every quadric $Q \in I_2(K_C)$ is invariant by the bielliptic involution. 

We define the Galois cover loci $\mathcal{BH}_{g,g', Gal}$ as the loci of  those bi(hyper)elliptic curves such that the $4:1$ map, that is the composition of the $2:1$ cover from $C$ to the (hyper)elliptic curve $C'$  with the $2:1$ map from that to $ \PP^1$, is a cyclic Galois cover ($g$ is the genus of $C$ and $g'$ is the genus of $C'$). 

 We prove that, if $g'\geq 2$ and $g \geq 3g'$, $\mathcal{BH}_{g,g', Gal}$ is nonempty. Then we use the explicit description of the space of holomorphic $1$-forms for cyclic Galois covers of the projective line to construct  invariant quadrics where we can compute the second fundamental form, using the second Gaussian map of the canonical bundle. \\ 
 
In general,  the study of the loci $\mathcal{BH}_{g,g', Gal}$ is one if the main tools in our computations, because for these $\Z/4\Z$ Galois covers we can explicitly determine a basis of $H^0(K_C)$ and hence also the quadrics that are invariant under the action of $\mathbb {Z}/4\mathbb {Z}$.

The second fundamental form of the Torelli map is a lifting of the second Gaussian map $\mu_2$ of the canonical bundle.  This fact is crucial for  this paper, since it allows in many cases to reduce the computation of the second fundamental form to the computation of $\mu_2$.

Gaussian maps are very important on their own and have been extensively studied. 
Just to mention a very important result, Wahl (\cite{wahl}) proved that for the canonically embedded curves which are
hyperplane sections of K3 surfaces, the first Gaussian (or Wahl) map of the canonical bundle is not surjective.  

In \cite{cfp} it is proven that for the canonically embedded curves which are
hyperplane sections of abelian surfaces,  the map $\mu_2$  is not surjective. 

In \cite{ccm} it is proven that for the general curve of genus $g \geq 18$, the map $\mu_2$ is surjective (hence of rank $7g-7$).  

On the other hand in \cite{colombo2008some} it is shown that for any hyperelliptic curve of genus $g \geq 3$, the rank of $\mu_2$ is  $2g-5$ and  for any trigonal curve of genus $g \geq 8$ the rank of $\mu_2$ is  $4g-18$.

A natural question is then to investigate the rank of the second Gaussian map $\mu_2$ on bielliptic or tetragonal curves.

In this paper we give an upper bound for the rank of $\mu_2$ for every bielliptic curve and a lower bound for this rank for the generic bielliptic curve.  Also this result is based on an explicit computation of the second Gaussian map  on  the locus $\mathcal{B}_{g,Gal}:= \mathcal{BH}_{g,1, Gal} $.  More precisely we have

\begin{teo}
For every bielliptic curve of genus $g$, we have 
 $$\rank \mu_2 \leq 5g-5.$$
The rank of the second Gaussian map for the general curve of the bielliptic locus satisfies the following bounds:

(1) If $g$ is odd then $ 2g-8 \leq \rank \mu_2$;

(2) If $g$ is even then $2g-9 \leq \rank \mu_2.$

\end{teo}

Finally we use the \verb|MAPLE| code reported in Appendix \ref{appendix_code} to compute the rank of the second Gaussian map on bielliptic curves of genus $5 \leq g \leq 30$ that are in $\mathcal{B}_{g,Gal}$.
These computations allow us to show the following 

\begin{teo}
The second Gaussian map on the bielliptic locus is generically injective if $5 \leq g \leq 8$, moreover it cannot be surjective for  $g \geq 14$. The general bielliptic curve of genus $8 \leq g \leq 30$ satisfies:
\begin{equation}
\rank \mu_2 \geq 2g-1.
\end{equation}
\end{teo}

The paper is organised as follows.

In Section \ref{sezione:preliminari} we first  recall the definitions of the second fundamental form of the period map, of the Gaussian maps and we explain the result of  \cite{colombo2000hodge} on the second fundamental form (see Theorem \ref{rhomu}). Then we explain the construction of the families of Galois covers that we need. 

In Section \ref{sezione:biellittiche} we show that the bielliptic locus is not totally geodesic. We first do it in the case $g \geq 5$, then we deal with the case $g=4$ separately. 

In Section \ref{sezione:bi-iperellittiche} we study the bihyperelliptic locus. We prove that for every genus $g \geq 3g'$  the locus $\mathcal{BH}_{g,g', Gal}$ is nonempty. Then we show that if $g \geq 3g'$ and $g' \geq 3$ the bihyperelliptic locus is not totally geodesic. Finally we prove that for $g'=2$ and $g \geq 6$ the bihyperelliptic locus is not totally geodesic. We also show that for $g \geq 8$ every irreducible component of the locus $\mathcal{BH}_{g,2,Gal}$ is not totally geodesic. 

In Section \ref{sezione:rango_biell} we show the upper and lower bounds for the rank of the second Gaussian map on the bielliptic locus. 

In Section \ref{section:computation_biellittic} we give the results of the computation of the rank of the second Gaussian map done with the \verb|MAPLE| script in low genus. 

In the Appendix we explain the \verb|MAPLE| script.

\medskip

{\bf Acknowledgement} We would like to thank Juan Carlos Naranjo, Pietro Pirola and Michael L\"onne for very  interesting and useful discussions on the subject.

\section{Notations and preliminary results}\label{sezione:preliminari}

\subsection{Second fundamental form}
Denote by $\mathcal{M}_g$ the moduli space of smooth complex algebraic curves of genus $g$, and by $\mathcal{A}_g$ the moduli space of principally polarized abelian varieties of dimension $g$. Call $j: \mathcal{M}_g \ra \mathcal{A}_g$ the period map, or Torelli map. Recall that the image of the Torelli map is called \emph{open Torelli locus} or \emph{open Jacobian locus}. We will denote it by $\tor^0$ and we will simply denote by $\tor$ its closure, called \emph{Torelli locus} or \emph{Jacobian locus}.
Both $\mathcal{M}_g$ and $\mathcal{A}_g$ are complex orbifolds, and $\mathcal{A}_g$ is endowed with a locally symmetric metric, the so-called Siegel metric, induced by the symmetric metric on the Siegel space $H_g$ of which ${\mathcal A}_g$ is the quotient by the action of $Sp(2g, \Zeta)$. Denote by $\nabla$ the corresponding metric connection. The Torelli map is an orbifold immersion outside the hyperelliptic locus (\cite{oort1979local}). Since for $g \geq 4$ the dimension of $\mathcal{M}_g$ is strictly smaller than the dimension of $\mathcal{A}_g$, it makes sense to study the metric properties of $\mathcal{M}_g$ with respect to the Siegel metric. More precisely, fix a non hyperelliptic curve $[C] \in \mathcal{M}_g$. Outside the hyperelliptic locus consider the short exact sequence of tangent bundles associated to the (orbifold) immersion $\mathcal{M}_g \ra \mathcal{A}_g$, evaluated at $[C]$:
\begin{equation}\label{short_exact_intro}
0 \ra T_{[C]} \mathcal{M}_g \xrightarrow{dj} T_{([JC],\Theta)} \mathcal{A}_g \xrightarrow{\pi} N_{[JC],\Theta} \ra 0,
\end{equation}
whose dual is 
\begin{equation}
0 \ra I_2(K_C) \ra S^2H^0(K_C) \ra H^0(2K_C) \ra 0.
\end{equation}

Denote by $II: S^2 T_{[C]} \mathcal{M}_g \longrightarrow  N_{([JC],\Theta)}$ the second fundamental form of the period map, and by $\rho : I_2(K_C) \ra S^2H^0(C,2K_C)\cong S^2H^0(C,T_C)^\vee$ its dual. When there is no risk of ambiguity, we will refer to $\rho$ as second fundamental form as well. 

\subsection{Gaussian maps}
Let $C$ be a smooth projective curve, $S: = C \times C$ and let $\Delta \subset S$ be the diagonal. 


Consider $p_1: S \ra C$ and $p_2: S \ra C$ the projections on the first and the second component respectively, take $L$ and $M$ line bundles on $C$ and define $L \boxtimes M := p_1^*L \otimes p_2^* M$. 
Consider the map given by restriction to the diagonal
\begin{equation*}
H^0(S, L \boxtimes M(-i\Delta)) \xrightarrow{\mu_{i,L,M}} H^0(S, L \boxtimes M(-i\Delta)_{|\Delta})\cong  H^0(C,  L \otimes M \otimes K_C^i ).
\end{equation*}
The map $\mu_{i,L,M}$ is called $i$-th \emph{Gaussian map}.  Denote $\mu_{i,L}:=\mu_{i,L,L}$. Since we will exclusively deal with Gaussian maps of order one and two with $L=M$, we write their explicit expressions. First notice that $\mu_{0,L}$ is the multiplication map 
$$H^0(C,L) \otimes H^0(C,L) \ra H^0(C, L^2),$$
which vanishes identically on $\Lambda^2H^0(C, L)$. Therefore its kernel which is  $H^0(S, L \boxtimes L(-\Delta))$ decomposes as $\Lambda^2H^0(C, L) \oplus I_2(L)$, where $I_2(L)$ denotes the kernel of the multiplication map $S^2H^0(C,L) \ra H^0(C, L^2)$.  The first gaussian map $\mu_{1,L}$ vanishes identically on symmetric tensors, hence one writes 

\begin{equation}\begin{split}
\mu_{1, L} : \Lambda^2H^0(C,L) \ra H^0(C,K_C \otimes L^2).
\end{split}\end{equation}
Take $l$ a local frame for $L$, $z$ a local coordinate on $C$ and $\sigma_1, \sigma_2 \in H^0(C,L)$. Locally write $\sigma_i = f_i(z) l$. Then the local expression of $\mu_{1,L}(\sigma_1 \wedge \sigma_2)$ is the following:
\begin{equation}\label{prima_gauss}
\mu_{1,L}(\sigma_1\wedge \sigma_2)= (f_i'(z) f_j(z) - f_i(z) \, f_j'(z))dz \otimes l^2.
\end{equation}
Now observe that $H^0(S, L \boxtimes L(-2\Delta))$ decomposes as the sum of $I_2(L)$ and the kernel of $\mu_{1,L}$.  Since $\mu_{2,L}$ vanishes identically on skew-symmetric tensors, we write 
\begin{equation}
\mu_{2,L}: I_2(L) \ra H^0(C,L^2 \otimes K_C^2).
\end{equation}

Let us describe it in local coordinates in the case in which $L$ is the canonical bundle. In this case we write $\mu_2:= \mu_{2,K_C}: I_2(K_C) \ra H^0(K_C^4)$. Fix a basis $\lbrace \omega_i \rbrace$ of $H^0(C,K)$, and assume that locally $\omega_i = f_i(z)dz$. Take a linear combination $\sum_{i,j} a_{ij} \omega_i \otimes \omega_j$ lying in  $I_2(K)$, so  that $\sum_{i,j} a_{ij} f_i(z)f_j(z)  = 0$. Then the local expression of $\mu_2(\omega_i \odot \omega_j)$ is the following:
\begin{equation}\label{mu2_loc}
\mu_2(\sum_{i,j} a_{ij} \, \omega_i \otimes \omega_j) = \sum_{i,j} a_{ij} f_i'(z)f_j'(z)(dz)^4.
\end{equation}

%

\subsection{Second fundamental form and second Gaussian map}

Let $C$ be a curve. Consider the second fundamental form $\rho: I_2(K) \ra S^2H^0(C,2K)$. 
We will state a result obtained in  \cite{colombo2000hodge} which gives an explicit expression of $\rho(Q)(v \otimes w)$, where $Q \in I_2(K_C)$ and  $v, w \in H^1(C, T_C)$ are Schiffer variations. 




\subsubsection{Schiffer variations}
Consider a curve $C$, a point $p \in C$ and the exact sequence 
\begin{equation}
0 \ra T_C \ra T_C(p) \ra T_C(p)_{|p} \ra 0.
\end{equation}
The coboundary map gives an injection $\delta: H^0(C,T_C(p)_{|p} ) \cong \C \hookrightarrow H^1(C, T_C)$. A Schiffer variation at $p$ is a generator of the image of $\delta$. 
Take a local coordinate $(U,z)$ centered at $p$, $b$ a bump function which is equal to 1 in a neighbourhood of $p$ and consider $\theta:= \frac{\bar{\partial}b}{z} \cdot \frac{\partial}{\partial z}\in A^{0,1}(T_C)$. The form $\theta$ is a Dolbeault representative of a Schiffer variation at $p$. More precisely, one can easily  check that the map 
$$\xi: T_pC \ra H^1(C, T_C), \ u:= \lambda \frac{\partial}{\partial z}(p) \mapsto \xi_u:= \lambda^2 [\theta]$$
does not depend on the choice of the local coordinate.

\subsubsection{The form $\eta_p$} Consider a curve $C$ of genus $g \geq 4$, and take a point $p \in C$. Consider the space $H^0(C,K_C(2p))$ of meromorphic $1$-forms on $C$ with a double pole on $p$, and notice that it maps injectively into $H^1(C \setminus p, \mathbb{C})$. By the Mayer-Vietoris sequence, the isomorphism $H^1(C, \mathbb{C}) \cong H^1(C \setminus p, \mathbb{C})$ holds, thus there is an injection:
\begin{equation}
j_p: \; H^0(C,K_C(2p)) \hookrightarrow H^1(C,\mathbb{C}).
\end{equation}
Remark that $H^0(C, K_C) \subset H^0(C, K_C(2p))$ maps to $H^{1,0}(C)$  through $j_p$. Moreover, since $h^0(C,K_C(2p))=g+1$, the preimage of the $(0,1)$-forms, $j_p^{-1}(H^{0,1}(C))$, has dimension $1$. Now, fix a local chart $(U,z)$ centered in $p$. Then there exists a unique element $\phi$ in this line such that its expression on $U\setminus p$ is
\begin{equation}
\phi := \bigg( \frac{1}{z^2}+h(z) \bigg) dz,
\end{equation}
where $h$ is a holomorphic function. One can define the form $\eta_p$ as follows:
\begin{equation}\begin{split}
\eta_p: T_pC &\longrightarrow H^0(C,K_C(2p)), \\
u=\lambda \frac{\partial}{\partial z}(p) &\longmapsto \eta_p(u)=\lambda \phi.
\end{split}\end{equation}
An easy computation shows that $\eta_p$ does not depend on the choice of the local coordinate. 

We have the following 

\begin{teo}[Colombo, Pirola, Tortora  ~\cite{colombo2000hodge}]\label{rhomu}
Let $C$ be a non-hyperelliptic curve of genus $g \geq 4$. Let $p,q \in C$ and $u \in T_p C$, $v \in T_q C$. Then:
\begin{equation}\label{rhomu_equazione}\begin{split}
&\rho (Q) (\xi _u \odot \xi _v) = -4 \pi i \eta _{p} (u)(v)Q(u,v), \\
&\rho(Q) (\xi_u \otimes \xi_u) = -2 \pi i \mu _2 (Q) (u^{\otimes 4}).
\end{split}\end{equation}
\end{teo}


In~\cite{colombo2010siegel}, Theorem \ref{rhomu} is used to compute the curvature of the restriction to $\mathcal{M}_g$ of the Siegel metric. Moreover, a more intrinsic description of the form $\eta_p$ is available: in~\cite{colombo2015totally} the authors proved that, as $p$ varies on the curve $C$, the forms $\eta_p$ glues to give a holomorphic section $\hat{\eta}$ of the line bundle $K_S(2\Delta)$. They also proved that the second fundamental form coincides with the multiplication by $\hat{\eta}$ and used this result to find constraints on the dimension of germs of totally geodesic submanifolds of $\mathcal{A}_g$ contained in the Jacobian locus.

\subsection{Families of Galois coverings}

Consider  a cover $f:X \rightarrow Y$, where $Y$ is a compact Riemann surface of genus $g' \geq 0$.
Let $t: = (t_1, \lds, t_r)$ be the branch locus of $f$, set $\ut := Y\setminus \{t_1, \lds, t_r\}$ and choose a base point $t_0 \in \ut$. The fundamental group $\pi_1(\ut, t_0 )$ is isomorphic to the group 

\begin{equation}\label{iso_fund_group}\Ga_{g',r}:= \langle \alpha _1, \beta _1, \dots, \alpha _{g'}, \beta _{g'}, \gamma _1, \dots, \gamma _r   \ | \ \prod _1 ^r \gamma _i \prod _1 ^{g'} \left[ \alpha _j, \beta _j \right] = 1 \rangle.
\end{equation}



Set $V := f^{-1}(\ut)$. Then
$f\restr{V} : V \ra \ut$ is an unramified Galois cover with Galois group $G$.  Since $\Ga_{g',r} \cong
\pi_1(\ut,t_0)$, we get an epimorphism $\theta : \Ga_{g',r} \ra G$. Recall that the stabiliser of a point under the action of $G$ is a cyclic group and denote by  $m_i$ the order of $\theta(\gamma_i)$, set  $\mm=(m_1, \lds, m_r)$. We can define a datum as follows.

\begin{defn}
\label{datum} 
  A \emph{datum} is a triple $\datum$, where $\mm :=(m_1, \ldots ,m_r)
  $ is an $r$-tuple of integers $m_i \geq 2$, $G$ is a finite group
  and $\theta : \Ga_{g',r} \ra G$ is an epimorphism such that $\theta
  (\ga_i)$ has order $m_i$ for each $i$.
\end{defn}
So, given a  Galois cover of $Y$
branched over $t$ and an isomorphism of $\pi_1(U_t, t_0) \cong \Gamma_{g',r}$, we obtain  a datum. The Riemann's existence theorem says that,  conversely,  a branch locus $t$ and a datum determine a covering of $Y$ up to
isomorphism (see e.g. \cite[Sec. III, Proposition 4.9]{miranda1995algebraic}).  
The genus $g$ of the Riemann surface $C$ is determined  by the Riemann-Hurwitz formula:
 \begin{equation*}\label{form.RH} 2g - 2 = |G|\left(2g'-2 +
      \sum_{i=1}^r \left(1 - \frac{1}{m_i}\right)\right).
  \end{equation*}

We  show that we can do this in families, namely that to any datum is associated a family of Galois covers of a compact Riemann surfaces $C'$ of genus $g'$. Let $\datum$ be a datum. 
Fix a compact oriented surface $\Sigma_{g'}$ of genus $g'$, $P= (p_1,...,p_r)$ an r-tuple of distinct points in  $ \Sigma_{g'}$,  and a point $p_0  \in  \Sigma_{g'}$, $p_0 \neq p_i$, $\forall i =1,...,r$. So the fundamental group $\pi_1(\Sigma_{g'} \setminus P, p_0)$ is isomorphic to $\Gamma_{g',r}$ and we fix such an isomorphism $\Psi: \pi_1(\Sigma_{g'} \setminus P, p_0)  \stackrel{\cong} \ra\Gamma_{g',r}$. 

Fix a point  $[C', t=(t_1,...,t_r), [f]]$  in the Teichm\"uller space  $\mathcal{T}_{g',r}$. So $C'$ is a compact Riemann surface of genus $g'$, $t=(t_1,...,t_r)$ is an r-tuple of distinct points in $Y$ and $[f]$ is the isotopy class of an orientation preserving  homeomorphism $f: C' \ra \Sigma_{g'}$. So we have an induced isomorphism  $f_*: \pi_1(C'-t, t_0) \cong \pi_1(\Sigma_{g'} \setminus P, p_0)$, where $t_0 = f^{-1}(p_0)$ and hence, composing with $\Psi$, an isomorphism $ \Psi \circ f_*: \pi_1(C' - t, t_0) \cong \Gamma_{g',r}$.  Using $\theta \circ \Psi \circ f_*$, we get a $G$-cover $C \ra C'$ branched at the points $t_i$ with local monodromies $m_1,\dots,m_r$.

 The curve $C$ is equipped with an isotopy class
of homeomorphisms to a fixed branched cover $\Sigma_g $ of $\Sigma_{g'}$.  Thus
we have a map ${\mathcal T}_{g',r} \ra {\mathcal T}_g \cong {\mathcal T}(\Sigma_g)$ to the
Teichm\"uller space of $\Sigma_g$.  The group $G$ embeds in the mapping
class group of $\Sigma_g$, denoted $Map_g$.  This embedding depends on
$\theta$ and we denote by $G(\theta) \subset Map_g$ its image.

  It
turns out that the image of ${\mathcal T}_{g',r}$ in ${\mathcal T}_g$ is exactly the set of
fixed points ${\mathcal T}_g^{G(\theta)}$ of the group $G(\theta)$.  We denote this
set by $\mathsf{T}\datum$. It is a complex submanifold of ${\mathcal T}_g$.  The image of
$    \mathsf{T}\datum   $ in the moduli space ${\mathcal M}_g$ is a $(3g'-3 + r)$-dimensional
algebraic subvariety that we denote by $\mathsf{M} \datum$.  See e.g. \cite{gonzalez1992moduli,catanese2011irreducibility,catanese2016genus} and \cite [Thm. 2.1]{broughton1990equisymmetric}
for more details.

We denote by $\zg(\mm,G,\theta)$ the closure in $\mathcal{A}_g$ of the image of $\mathsf{M}\datum$ via the Torelli map.   By the above, it is an algebraic subvariety of $\mathcal{A}_g$  of dimension $3g'-3+r$.

It can happen that different data give rise to the same subvariety of $\mathcal{M}_g$.  This depends on the choice of the isomorphism $ \Psi: \pi_1(\Sigma_{g'}, p_0) \cong \Gamma_{g',r}$. The change of the isomorphism is given by the action of the mapping class group $Map_{g',[r]}:= \pi_0(Diff^+(\Sigma_{g'} \setminus \{x_1,...,x_r\}))$. For a description of this action see e.g. ~\cite{penegini2015surfaces}. The group $Map_{g',[r]}$ injects in $Out^+(\Gamma_{g',r})$, hence we get an action of the mapping class group on the set of data up to inner automorphisms. Clearly also the group $Aut(G)$ acts on the set of data: $\phi \cdot \datum := ({\bf{m}}, G, \phi \circ \theta)$. 
The orbit by these actions are called Hurwitz equivalence classes. Data in the same class yield the same subvarieties $\mathsf{M}\datum$ and hence the same $\zg(\mm,G,\theta)$. 

In \cite{frediani2015shimura}, \cite{frediani2016shimura} it is proven that if for $[C] \in \mathsf{M}\datum$ one has
$$ (*) \ \ \dim(S^2 H^0(C, K_C))^G= 3g'-3+r,$$
then $\zg(\mm,G,\theta)$ is a Shimura subvariety of ${\mathcal A}_g$.

\section{The bielliptic locus is not totally geodesic}\label{sezione:biellittiche}
A curve $C$ is called bielliptic if it is a double cover of an elliptic curve $E$. 

Denote by ${\mathcal B}_g$ the bielliptic locus. We start by recalling some elementary results on $\mathcal{B}_g$, which will be useful to clarify the set-up and basic properties of this locus. First of all from Castelnuovo-Severi inequality (see for instance~\cite{accola2006topics}) it follows immediately that if $g \geq 4$, the intersection between $\mathcal{B}_g$ and the hyperelliptic locus is empty. 
Using Riemann-Hurwitz formula, the dimension of the bielliptic locus is $2g-2$. It is known that the bielliptic locus $\mathcal{B}_g$ is irreducible (see e.g.~\cite{bardelli1999bielliptic}). 

With a simple dimension count, we remark that if $C$ is a bielliptic curve, every quadric $Q \in I_2(K_C)$ is invariant under the bielliptic involution $\sigma$. In fact, $H^0(C, K_C) \cong H^0(E, K_E) \oplus H^0(C, K_C)^-$, where $H^0(C, K_C)^-$ denotes the anti-invariant subspace by the action of $\sigma$ and $H^0(E, K_E) \cong H^0(C, K_C)^+ $ is the invariant subspace. 
Hence we have $\big( S^2 H^0(C,K_C) \big)^- \cong H^0(E,K_E) \otimes H^0(C,K_C)^-$, so $\operatorname{dim} \big( S^2 H^0(C,K_C) \big)^- =g-1$. Moreover the space $H^0(C, 2K_C)^+$ can be identified with the cotangent space at $[C]$ of the bielliptic locus ${\mathcal B}_g$, hence $\operatorname{dim}H^0(C,2K_C)^+= 2g-2$, thus $\operatorname{dim}H^0(C,2K_C)^- =  g-1$. Since the multiplication map is $\sigma$-equivariant, we have the following exact sequence
\begin{equation*}
0 \rightarrow I_2(K_C)^- \rightarrow (S^2 H^0(C,K_C))^- \rightarrow H^0(C,2K_C)^- \rightarrow 0,
\end{equation*}
thus  $I_2(K_C)^- = (0)$ and $I_2(K_C) = I_2(K_C)^+$.

We prove that the bielliptic locus is not totally geodesic. We will consider separately the case $g=4$ and the case $g \geq 5$. Remark that in case $g=3$ the bielliptic locus is already known to be totally geodesic (see example $(2)$ in~\cite[\S 3.1]{frediani2016shimura}).
Observe that in \cite[Theorem 4.2, Remark 4.3] {colombo2015totally} it is proven that the maximal dimension of a germ of a totally geodesic submanifold of ${\mathcal A}_g$ contained in the Torelli locus and passing through some $k$-gonal curve is $2g+k-4$ (it is $2g-2+k/2$ for the general $k$-gonal curve). For $k=4$ this bound is equal to $2g$. Since the bielliptic locus has dimension $2g-2$ it is natural to ask whether it is totally geodesic. 

We start studying the case $g \geq 5$. The trick for the proof is similar to the one used in  \cite[Theorem 4.1]{colombo2015totally} to obtain a bound for the dimension of germs
of totally geodesic submanifolds contained in the Jacobian locus. 
\begin{teo}\label{biell_nontotgeo_g5}
The bielliptic locus is not totally geodesic if $g \geq 5$.
\end{teo}
\begin{proof}
Consider a bielliptic curve $C$ of genus $g \geq 5$. It admits a $\mathfrak{g}_4^1$ obtained by the composition of the $2:1$ map $C \ra E$ with the $2:1$ covering $E \ra {\PP}^1$. We call $|F|$ this $\mathfrak{g}_4^1$. Fix a basis for $H^0(F)=\langle x_1, x_2 \rangle$ and consider the adjoint linear series given by $|K-F|$. By Riemann-Roch we have $h^0(K-F) =  g-3 \geq 2$ for $g \geq 5.$
So we can find a pencil $\langle t_1, t_2 \rangle \subseteq H^0(K-F)$. We construct a quadric in $I_2(K_C)$  in the following way:
\begin{equation}\label{quadric_biell}
Q := x_1 t_1 \odot x_2 t_2 - x_1 t_2 \odot x_2 t_1 \in I_2(K_C) =  I_2(K_C)^+.
\end{equation}
%

Consider two points $p_1, p_2= \sigma(p_1) \in C$ lying on the same fiber over a point $p \in E$ that are non critical for $|F|$, $|K-F|$, that lie outside the base locus of $|K-F|$ and such that $p_1 \neq p_2$. Notice that $\xi_{p_1} + \xi_{p_2} \in H^1(C,T_C)^+$. Using the $\sigma$-equivariance of both the map $\rho$ and the quadric $Q$ one has:
\begin{equation}
\rho(Q)((\xi _{p_1} + \xi _{p_2}) \odot (\xi _{p_1} + \xi _{p_2})) = 2\rho(Q)(\xi _{p_1} \odot \xi _{p_1}) + 2 \rho(Q)(\xi _{p_1} \odot \xi _{p_2}).
\end{equation}

Notice that, by construction, the quadric $Q$ vanishes when evaluated over points lying over the same fiber of the map induced by $|F|$, hence $Q(p_1, p_2) =0$. So, by Theorem \ref{rhomu}, 
$\rho(Q)((\xi _{p_1} + \xi _{p_2}) \odot (\xi _{p_1} + \xi _{p_2})) = 2\rho(Q)(\xi _{p_1} \odot \xi _{p_1}) = -4 \pi i \mu_2(Q)(p_1).$

From~\cite[Lemma 2.2]{colombo2008some}, the second Gaussian map decomposes as:
\begin{equation}
\mu_2(Q) = \mu_{1,F}(x_1 \wedge x_2) \mu_{1,K-F}(t_1 \wedge t_2).
\end{equation}
Furthermore $\mu_{1,F}(x_1 \wedge x_2)$ and $\mu_{1,K-F}(t_1 \wedge t_2)$ vanish, respectively, only over the set of  critical points of $|F|$ and $|K-F|$ and on the base locus of $|K-F|$. By our choice of the point $p_1$  we can conclude that 
\begin{equation} 
\label{nonzero}
\rho(Q)((\xi _{p_1} + \xi _{p_2}) \odot (\xi _{p_1} + \xi _{p_2})) = -4 \pi i \mu_2(Q)(p_1) \neq 0.
\end{equation}
This concludes the proof since if we denote by $\tilde{\rho}$ the dual of the second fundamental form of the bielliptic locus in ${\mathcal A}_g$, one immediately sees that for any bielliptic curve $C$ and any quadric $Q \in I_2(K_C)^+$ and for any $u,v \in H^1(T_C)^+$, one has $\tilde{\rho}(Q)(u \odot v) = \rho(Q)(u \odot v)$ (see \cite[5.2]{colombo2015totally}).

Hence \eqref{nonzero} shows that $\tilde{\rho}(Q)( (\xi _{p_1} + \xi _{p_2}) \odot (\xi _{p_1} + \xi _{p_2})) \neq 0$, so the bielliptic locus is not totally geodesic in ${\mathcal A}_g$ if $g \geq 5$. 
\end{proof}

In the following we consider the case $g=4$. We stress that we have considered separately cases $g \geq 5$ and $g=4$ because in the latter case $h^0(K-F)=1$, so one cannot construct a quadric of rank $4$ using $|F|$ and $|K-F|$. In case $g=4$ we will fix the problem by using the two $\mathfrak{g}_3^1$'s of $C$ that we denote by  $|M|$ and $|K-M|$.

%
%
%
\begin{prop}\label{invol_switch}
The generic bielliptic curve of genus $g=4$ admits two different $\mathfrak{g}_3^1$'s switched by the bielliptic involution.
\end{prop}
\begin{proof}
Recall that a bielliptic curve of genus $g=4$ is non-hyperelliptic. Its canonical model is a smooth complete intersection of a quadric and a cubic in $\PP^3$. From Brill-Noether theory it follows that a $\mathfrak{g}_3^1$ is  cut out by the lines of a ruling of the quadric containing the canonical model of $C$ (see e.g.~\cite[page 206]{arbarello1985geometry}). 

Assume that there exists only one $\mathfrak{g}_3^1$ $|M|$. Then $|M| = |K-M|$ and denote by $\phi: C \ra {\mathbb P}^1$ the map  induced by $|M|$. The involution $\sigma$ acts on the linear series $|M|$, hence we have a commutative diagram 
\begin{equation}
\label{phi} 
\begin{tikzcd}
C \arrow[swap]{d} {\phi} \arrow{r}{\sigma} & C\arrow{d}{\phi} \\
 \mathbb{P}^1 \arrow{r}{\bar{\sigma}} &    \mathbb{P}^1\\
\end{tikzcd}
\end{equation}
The biholomorphism $\bar{\sigma}$ can't be the identity, since otherwise $\sigma$ would acts on the fibers of $\phi$ which is $3:1$ and this is clearly impossible.
Hence $\bar{\sigma}$ has order two and so it has exactly two fixed points $s,t$.  

The involution $\sigma$ has exactly six fixed points on $C$ and if a point $p$ in $C$ is fixed by $\sigma$, then $\phi(p_1)$ is fixed by $\bar{\sigma}$. Since $\bar{\sigma}$ has only two fixed points $s,t$, the fixed points by $\sigma$ must coincide with the two fibers $\phi^{-1}(s)$ and   $\phi^{-1}(t)$, which therefore do not contain any critical point of $\phi$. Denote by $\{p_1,p_2,p_3\} = \phi^{-1}(s)$, $\{p_4,p_5,p_6\} = \phi^{-1}(t)$. Then $|M| = |p_1+p_2+p_3| = |p_4+p_5 +p_6|$. So if we denote by $Nm: Pic^0(C) \ra Pic^0(E)$ the norm map of the double cover $\pi: C \ra E$, then the element $\OO_C(p_1+p_2+p_3-p_4-p_5-p_6) \in Pic^0(C)$ is sent by $Nm$ to $\OO_E(q_1+q_2+q_3-q_4-q_5-q_6) \in Pic^0(E)$, where $q_i = \pi(p_i)$. So the critical values of $\pi$ are $q_1,...,q_6$ and $q_1+q_2+q_3 \sim q_4 + q_5 + q_6$ in $E$. Therefore it suffices to choose as critical values 6 points not satisfying this property. This is clearly the generic case. 

So we have proven that the generic bielliptic curve of genus $4$ has two distinct $\mathfrak{g}_3^1$'s $|M|$ and  $|K-M|$.  Denote by $\phi$, $\psi$ the two maps induced by $|M|$ respectively $|K-M|$. 
We want to show that $|M|$ and $|K-M|$ are exchanged by the bielliptic involution. Assume that this is not the case. 
Then we have a diagram as in \eqref{phi} for $\phi$ (and a similar one also for $\psi$). But we have proven that this does not happen generically. 

We can prove something slightly stronger, namely that for any  bielliptic curve admitting two different $\mathfrak{g}_3^1$'s, these are exchanged by $\sigma$. 
In fact if this is not the case, we have a diagram as in \eqref{phi} for $\phi$ and also one for $\psi$.   The fixed points of $\sigma$ coincide with two non critical fibers of $\phi$ and also with  two non critical fibers of $\psi$. Hence we have  $|2M| = |2(K-M)|$, equivalently $2(M - M') \sim 0$, where $M' = K-M$. But this is impossible (see e.g.  \cite[Proposition 3.4]{artebani2016semiample}). 
\end{proof}

\begin{teo}\label{biell_nontotgeo_g=4}
The bielliptic locus is not totally geodesic if $g=4$.
\end{teo}
\begin{proof}
Let $C$ be a generic  bielliptic curve of genus $4$. From Proposition \ref{invol_switch}, $C$ admits two different $\mathfrak{g}_3^1$'s, $|M|$ and $|K-M|$, switched by the bielliptic involution. That is, the following diagram holds:
\begin{equation}
\label{phi-psi} 
\begin{tikzcd}
C \arrow[swap]{d} {\phi} \arrow{r}{\sigma} & C\arrow{d}{\psi} \\
 \mathbb{P}^1 \arrow{r}{\bar{\sigma}} &    \mathbb{P}^1\\
\end{tikzcd}
\end{equation}
where $\phi$ and $\psi$ are the maps induced by $|M|$ respectively  $|K-M|$. 

We have the following 
\begin{claim}
There exists at most one fixed point for $\sigma$ which is critical for $\phi$ (and hence also for $\psi$).
\end{claim}
To prove the claim, we argue by contradiction: suppose that there are two different points, $p$ and $p'$ fixed by $\sigma$ and critical for $\phi$ and $\psi$. Call $w=\phi(p)$ and $w'=\phi(p')$. Consider the pullbacks:
\begin{equation}\begin{aligned}
&|L| = |\phi^*(w)| =|2p+q|,\qquad &|K-L|=  |\psi^*(\bar{\sigma}(w))|=|2p+q'|,\\
&|L| = |\phi^*(w')| =|2p'+r|, &|K-L| = |\psi^*(\bar{\sigma}(w'))|=|2p'+r'|.
\end{aligned}\end{equation}
The above equations imply that $|K-2L| = |q'-q| = |r'-r|$. So one gets $|q'+r| = |q+r'|$. This is absurd, since the curve $C$ is not hyperelliptic.

Using the claim, choose a point $p_1 \in C$ which is fixed by the involution $\sigma$ and regular for $\phi$ and $\psi$. 
\begin{claim}
The Schiffer variation $\xi _{p_1}$ is invariant under the bielliptic involution~$\sigma$.
\end{claim}
\begin{proof}
Recall that in a  local coordinate $(U,z)$ around $p_1$ a Schiffer variation in $p_1$ has a Dolbeault representative  
\begin{equation}\label{schiffer_invar}
\theta= \frac{\bar{\partial}b_{p_1}}{z} \cdot \frac{\partial}{\partial z},
\end{equation}
where $b_{p_1}$ is a bump function which is equal to 1 in a small neighbourhood containing $p_1$.

By assumption, $\sigma(p_1)=p_1$. Linearising the action of $\sigma$ (see e.g. \cite[Corollary 3.5]{miranda1995algebraic}), we find a local chart $(U,z)$ in a neighbourhood of $p_1$ such that $\sigma(z)=-z$.   We can choose the bump function $b_{p_1}$ to be invariant under $\sigma$. So one immediately gets that $\xi_{p_1}$ is invariant by the action of  $\sigma$.

\end{proof}

Thus we have an invariant Schiffer variation. 
We construct an invariant quadric using the  $\mathfrak{g}^1_3$'s:
\[ Q := x_1 t_1 \odot x_2 t_2 - x_1 t_2 \odot x_2 t_1 \in I_2(K_C)^{\sigma}, \]
where $H^0(M)=\langle x_1, x_2 \rangle$ and $H^0(K-M) = \langle t_1, t_2 \rangle$.
By Theorem \ref{rhomu} we have 
$$\rho(Q)(\xi_{p_1} \odot \xi_{p_1}) = -2 \pi i \mu_2(Q)(p_1),$$
and
$$ \mu_2(Q) = \mu_{1,|M|}(x_1 \wedge x_2) \mu_{1,|K-M|}(t_1 \wedge t_2). $$
Since $|M|$ and $|K-M|$ are base point free, by construction $\mu_{1,|M|}(x_1 \wedge x_2)$ and $\mu_{1,|K-M|}(t_1 \wedge t_2)$ vanish, respectively, over critical points of $|M|$ and $|K-M|$. Since $p_1$ is non-critical for the two maps, we conclude that $$\rho(Q)(\xi_{p_1} \odot \xi_{p_1}) = -2 \pi i \mu_2(Q)(p_1) \neq 0,$$
and this concludes the proof as above. 
\end{proof}

\section{The bihyperelliptic locus is not totally geodesic}\label{sezione:bi-iperellittiche}

A curve $C$ is called bihyperelliptic if it is a double cover of a hyperelliptic curve $C'$, with $g' = g(C') \geq 2$:
\begin{equation}
C \xrightarrow{2:1} C' \xrightarrow{2:1} \mathbb{P}^1.
\end{equation}

We will denote with $\mathcal{BH}_{g,g'}$ the locus of bihyperelliptic curves $C \ra C'$ where $g(C)=g$ and $g(C')=g'$. In \cite{ballico2007geometry} it is proven that if $g \geq 4g' +2$ the locus $\mathcal{BH}_{g,g'}$ is irreducible and unirational of dimension $2g-2g'+1$. Note that in the introduction to  \cite{ballico2007geometry} it is incorrectly stated that the dimension of $\mathcal{BH}_{g,g'}$ is $g-g' + 2[(g-g')/2] +1$ (the integer part should not be there).



%

\begin{remark}
Notice that if $g \geq 3g'$, $g'\geq 2$, by the Castelnuovo-Severi inequality (see e.g. \cite{accola2006topics}), a bihyperelliptic curve is not hyperelliptic.  
\end{remark}

We want to show that the bihyperelliptic locus is non totally geodesic when $g \geq 3g'$.  For $g \geq 4g' +2$  it is irreducible \cite{ballico2007geometry}.  For $3g' \leq g < 4g' +2$ we prove that there exists at least one irreducible component that is not totally geodesic. 

To this purpose we will consider those bihyperelliptic covers  $\phi: C \ra C'$ such that the composition $\psi: C \ra \PP^1$ of $\phi$ with the $2:1$ cover $\pi: C' \ra \PP^1$, is cyclic (with Galois group $\Z/4$).

So  denote by 
\begin{equation*}\begin{aligned}
\mathcal{BH}_{g,g', Gal}:= \Bigg\lbrace [C] \in \mathcal{M}_g  \ | \    \text{ there exists a 2:1 cover} \  \phi: C \ra C' \ , \text{where}\\
 \  C' \ \text{is a hyperelliptic curve of genus}  \ g' \ \text{and} \ \exists \  \text{a commutative diagram}\\
\begin{tikzcd}
C \arrow[swap]{dr} {\psi} \arrow{rr}{\phi} && C' \arrow{dl}{\pi} \\
 & \mathbb{P}^1
\end{tikzcd} \quad \\
\mbox{where } \psi: C \rightarrow \mathbb{P}^1 \mbox{ is a Galois cover with group } \Z/4\Z \Bigg\rbrace.
\end{aligned}\end{equation*}

\begin{teo}
\label{Monodromiesbihyper} 
For every $g \geq 3g'$ the locus $\mathcal{BH}_{g,g', Gal}$ is not empty. 
\end{teo}

\begin{proof}

Assume first of all that such a diagram exists. Denote by 
$r_{4}$ and $r_{2}$ the number of critical values of $\psi$ with multiplicity $4$ and respectively 2. Using Riemann-Hurwitz formula one gets:
\begin{equation}\label{eq_mono_biiper}
2g + 6 = 3r_{4} + 2r_{2}.
\end{equation}
Assume that $z$ is a generator of $\Zeta/4$ and consider the map 
$$\phi: C \ra C' = C/\langle z^2 \rangle.$$ 
We remark that $\phi$ and $\psi$ have the same critical points. In fact,  the stabilizer in $\Z/4$ for $p \in \operatorname{Crit}(\psi)$ is either $\langle z \rangle$ or $\langle z^2 \rangle$. In both cases $p$ is fixed by the action of $\Zeta/2\Zeta = \langle z^2 \rangle$ as well. So, applying Riemann-Hurwitz formula for $\phi: C \ra C'$ one gets:
\begin{equation}
2g-4g'+2 = r_{4} + 2 r_{2}.
\end{equation}
Substituting in equation (\ref{eq_mono_biiper}), one obtains:
\begin{equation}\label{mono_biiper}
    r_2 = g-3g' \geq 0;\qquad \qquad 
    r_{4} = 2g'+2.
\end{equation}

This proves that the only possible choice for the vector of the multiplicities of the critical values of $\psi$  is $m = (4^{2g'+2},2^{g-3g'})$, where we use the compact notation meaning that $m_i = 4$, $\forall i \leq 2g'+2$ and $m_i = 2$, $\forall i >2g'+2$.

The monodromy of $\psi$ is given by an epimorphism $\theta: \Gamma_{0,m} \twoheadrightarrow \mathbb{Z}/4\mathbb{Z}$, where $m = r_2 + r_4$, 
$\Gamma_{0,m} := \langle \gamma_1 ,\dots, \gamma_{r_4}, \gamma_{r_4+1} ,\dots, \gamma_{r_4+r_2} \; |  \gamma_i^{m_i}=1, \; \gamma_1 \cdots \gamma_m=1 \rangle.$

Notice that $\theta(\gamma_i)=z^2$ for $r_4+1 \leq i \leq r_4+r_2$, while $\theta(\gamma_i)$ can be either $z$ or $z^3$ for $1 \leq i \leq r_4$. Call $s_1$ and $s_3$ the number of $z$ and $z^3$ occurring, respectively ($s_1+s_3=r_4$). Then we have:  
\begin{equation}\label{condizione_dimobiiper}
s_1 + 3s_3 + 2r_2 \equiv 0 \quad \operatorname{ mod} 4.
\end{equation}

So to prove that $\mathcal{BH}_{g,g', Gal}$ is non empty if $g \geq 3g'$ we construct a cyclic $4:1$ cover $\psi: C \ra {\PP}^1 = C/\Zeta/4$ given by an epimorphism $\theta$ as above with $r_2 = g -3g'$, $r_4 = s_1 + s_3 = 2g'+2$ satisfying  \eqref{condizione_dimobiiper}. Notice that condition  \eqref{condizione_dimobiiper} turns out to be equivalent to $s_1 \equiv g+1$ (mod 2). 

We verify now that the cover $\psi$ that we have constructed admits a factorisation via a $2:1$ cover of a curve of genus $g'$. 

In fact, consider the map $\phi: C \ra C'= C/\langle z^2 \rangle$. Applying Riemann Hurwitz one gets 
$$2g - 2 = 2(2g(C') -2) + r_{4} + 2 r_{2} = 2(2g(C') -2) + 2g' + 2 + 2(g -3g'),$$
hence $g(C') = g'$. This concludes the proof. 
\end{proof}

Now we  prove that the bihyperelliptic locus is not totally geodesic when $g \geq 3g'$. We will consider separately cases $g'=2$ and $g' \geq 3$. As in the bielliptic case, we apply Theorem \ref{rhomu} to write the second fundamental form explicitly.


To do this we restrict to $\mathcal{BH}_{g,g',Gal}$ and we construct an invariant quadric using the decomposition of $H^0(K_C)$ induced by the structure of cyclic Galois cover of $\PP^1$ (see e.g. ~\cite{moonen2010special}, from which we borrow the notations).

Let us fix a datum $\datum$, where $G = \Z/m$  and $$\theta: \Gamma_{0,r} = \langle \gamma_1,...,\gamma_r \ | \prod_{i=1,...,r} \gamma_i = 1 \rangle \ra \Z/m.$$
Denote by $a= ([a_1],..., [a_r]) := (\theta(\gamma_1), ...., \theta(\gamma_r))$, where $a_i \in \{1,...,m-1\}$, $\sum_{i =1,...,r} a_i \equiv 0$ mod $m$.

In this case the mapping class group $Map_{0, [r]}$ is the braid group and it acts on the vector $a=([a_1],..., [a_r])$ by permuting the entries (see e.g. \cite{penegini2015surfaces}), while the group $Aut(\Z/m) \cong (\Z/m)^*$ acts by multiplication on the entries of $a$ by an invertible element $[n] \in  ( \Z/m)^*$. As usual we consider these data only up to these two actions.

\begin{teo}\label{biiper_nontot_g3}
If $g \geq 3g'$ and $g' \geq 3$, the bihyperelliptic locus is not totally geodesic.
\end{teo}
\begin{proof}
By the previous result, there exists a curve $[C] \in \mathcal{BH}_{g,g',Gal}$. Take the Galois cover $\psi: C \ra \PP^1$. 
Consider the decomposition of the space of holomorphic one forms in eigenspaces, $H^0(K_C) = V_1 \oplus V_2 \oplus V_3$, and observe that, independently from the chosen monodromy, $\dim V_2 = g' \geq 3$, since $V_2 \cong H^0(C', K_{C'})$. We use  three forms $\omega_{2,0},\omega_{2,1},\omega_{2,2} \in V_2$  to construct a quadric as follows:
\begin{equation}\label{quadrica_1}
Q:= \omega_{2,0} \odot \omega_{2,2} - \omega_{2,1} \odot \omega_{2,1}. 
\end{equation}
Let us now explain how the forms $\omega_{2,0},\omega_{2,1},\omega_{2,2} \in V_2$ are obtained. Suppose that the cyclic covering $\psi$ is given by the equation $y^4 = \prod_{i=1,...,r} (x-t_i)^{a_i}$.  We can assume that $t_1 =0$.

One can prove  that the forms
$$\omega_{2,0} := y^2 \prod_{i=1}^{r} (x-t_i)^{l(i,2)}dx, \ l(i,2)  = \left \lfloor{\frac{-a_i}{2}}\right \rfloor, \ \omega_{2,1} := x \omega_{2,0}, \  \omega_{2,2} := x^2 \omega_{2,0}$$ are holomorphic (see e.g.  ~\cite[Lemma (2.7)] {moonen2010special}).

So, by construction,  the quadric $Q$ lies in $I_2(K_C)$  and it is  also invariant under the whole $\Zeta/4\Zeta$. In fact, let $\zeta_4$ be a primitive fourth root of unity,
\begin{equation*}
g^*(Q) = \zeta_4^2 \cdot \omega_{2,0} \odot \zeta_4^2 \cdot  \omega_{2,2} - \zeta_4^2 \cdot  \omega_{2,1} \odot \zeta_4^2 \cdot \omega_{2,1}=\zeta_4^{4} \cdot Q = Q.
\end{equation*}
Moreover, it is an easy computation to show that $Q(p_1,p_2)=0$ for every couple of points $p_1$, $p_2$ lying in the same fiber over a point in $\mathbb{P}^1$.

Recall that  the tangent space to the bihyperelliptic locus at a point $[C]$ is isomorphic to the invariant subspace $H^1(C, T_C)^{\Zeta/2}$.  We take an invariant tangent vector given by the sum of two Schiffer variations $\xi_{p_1} + \xi_{p_2}$, where $p_1$ and $p_2$ are distinct points in the same fiber of the map $\phi: C \ra C/\Zeta/2$. Since $Q(p_1,p_2) =0$, we apply Theorem~\ref{rhomu} and we get 
\begin{equation*}
\rho(Q)((\xi _{p_1} + \xi _{p_2}) \odot (\xi _{p_1} + \xi _{p_2})) = 2\rho(Q)(\xi _{p_1} \odot \xi _{p_1}) = -4 \pi i \mu_2(Q)(p_1).
\end{equation*}

So it suffices to choose a point $p_1 \in C$ not fixed by the covering involution $\sigma$ of $\phi$ where $\mu_2(Q) $ does not vanish (clearly $p_2 = \sigma(p_1)$). This is possible since $\mu_2(Q)$ is a non zero section of $H^0(4K_C)$. To prove that it is non zero, take a local coordinate $z$ and write $\omega_{2,0} = f(z)dz$, $x = k(z)$, then $\omega_{2,1}=k(z)\omega_{2,0}$ and $\omega_{2,2} = (k(z))^2 \omega_{2,0}$. We have $Q = f(z)dz \odot (k(z))^2f(z)dz - k(z) f(z) dz \odot k(z) f(z) dz$. So one computes  
\begin{equation}\label{mu2neq0}\begin{split}
\mu_2(Q)&= f'(z)(k(z)^2f(z))' - (k(z)f(z))'(k(z)f(z))'=\\
&-k'(z)^2 f(z)^2 \neq 0.
\end{split}\end{equation}

\end{proof}

It remains to consider the case $g'=2$.

\begin{teo}\label{biiper_nontot_g2}
If $g' =2$ and $g\geq 6$ the bihyperelliptic locus is not totally geodesic.
\end{teo}
\begin{proof}
We will show that for every $g \geq 6$, there exists a curve $[C] \in \mathcal{BH}_{g,2,Gal}$, a quadric $Q \in I_2(K_C)^{\Z/2}$ and a tangent vector $v \in T_{[C]}  \mathcal{BH}_{g,2} = H^1(C, T_C)^{\Z/2}$ such that $\rho(Q)(v \odot v) \neq 0$.

Consider a datum $(\mm, \Z/4, \theta)$. By \ref{mono_biiper} $r_2 = g-6$, $r_4=6$,  hence ${\bf{m}}=(4^6,2^{g-6})$.  Here we use a compact notation, i.e. $m_i = 4$ for $i \leq 6$, $m_i = 2$, for $i =7,...,g$. We will use an analogous compact notation for the monodromy $a=([a_1],...,[a_g])= (\theta(\gamma_1),..., \theta(\gamma_g))$. 

If $g$ is even take $ a=([1]^5,[3],[2]^{g-6})$. If $g$ is odd take $ a=([1]^6,[2]^{g-6})$. In both cases we construct the corresponding family of cyclic covers of $\PP^1$. If $\psi: C \rightarrow \PP^1$ is an element of the family, on can compute the dimensions $d_i$ of the eigenspaces $V_i \subset H^0(C, K_C)$ for the action of the cyclic Galois group $\Z/4$ using  ~\cite[Lemma (2.7)] {moonen2010special} and obtain that  in the case $g$ even,
$$d_1 = \frac{g}{2}; \qquad d_2 = 2; \qquad d_3 = \frac{g}{2}-2;$$

in the case $g$ odd, 
$$d_1 =\frac{g+1}{2}; \qquad d_2 = 2; \qquad d_3 = \frac{g-1}{2}-2.$$

In both cases, $d_1 \geq 3$, so we can construct the $\Zeta/2$-invariant quadric $Q = \omega_{1,0} \odot \omega_{1,2} - \omega_{1,1} \odot \omega_{1,1} =  x^2 \, \omega_{1,0} \odot \omega_{1,0} - x \, \omega_{1,0} \odot x \, \omega_{1,0}$, where

\begin{equation}
\label{omega10}
\omega_{1,0} = y \prod_{i=1}^g (x-t_i)^{l(i,1)} dx, \  \  \  l(i,1) =  \left \lfloor{\frac{-a_i}{4}}\right \rfloor.
\end{equation}
Clearly $Q \in I_2(K_C)^{\Z/2}$, so if we take two distinct points $p_1,p_2$ belonging to the same fiber of the map $\phi: C \ra C/\Z/2 = C'$, we obtain an invariant tangent vector $\xi_{p_1} + \xi_{p_2} \in H^1(C, T_C)^{\Z/2}$. 
Since $p_1, p_2$ are in the same fiber of the map $\psi: C \ra \PP^1$, $x(p_1) = x(p_2)$, hence $Q(p_1, p_2) = 0$.  

Thus, applying Theorem~\ref{rhomu},   we get 
\begin{equation*}
\rho(Q)((\xi _{p_1} + \xi _{p_2}) \odot (\xi _{p_1} + \xi _{p_2})) = 2\rho(Q)(\xi _{p_1} \odot \xi _{p_1}) = -4 \pi i \mu_2(Q)(p_1) \neq 0,
\end{equation*}
if we choose $p_1$ not in the zero locus of $\mu_2(Q)$. This is possible because again one can easily verify that $\mu_2(Q) \neq 0$. 
This concludes the proof. 
\end{proof}

With a little more effort one can prove the following 

\begin{teo}
\label{bihygal}
If $g \geq 8$, every irreducible component of the locus $\mathcal{BH}_{g,2,Gal}$ is not totally geodesic. 
\end{teo}
\begin{proof}
First take a curve $[C] \in \mathcal{BH}_{g,2,Gal}$ and take the cyclic cover $\psi: C \ra \PP^1 = C/G$, where $G \cong \Z/4$. We start computing the possible monodromies for $\psi$. In this case, the orders of ramifications are necessarily ${\bf m}=(4^6,2^{g-6})$.   One can check that, up to Hurwitz equivalence, there are only four distinct families given by the following monodromy data in the sense of Definition  \ref{datum} (for simplicity we only write $(\theta(\gamma_1),...,\theta(\gamma_r))$ for a datum, since the group is $\Z/4$ and the orders of the elements are given by ${\bf m}=(4^6,2^{g-6})$).

1. If $g$ is even then:
\begin{equation}\begin{split}
&(a) \quad a=([1]^5,[3],[2]^{g-6}) \mbox{ or}\\
&(b) \quad a=([1]^3,[3]^3,[2]^{g-6}),
\end{split}\end{equation}

and the dimensions of the eigenspaces $V_1$, $V_2$, $V_3$ are the following (see ~\cite[Lemma (2.7)] {moonen2010special}):

$$(a) \quad \; \; \qquad \qquad d_1 = \frac{g}{2}; \qquad d_2 = 2; \qquad d_3 = \frac{g}{2}-2;$$
$$(b) \qquad \qquad d_1 = \frac{g}{2}-1; \qquad d_2 = 2; \qquad d_3 = \frac{g}{2}-1.$$

2. If $g$ is odd then:

\begin{equation}\begin{split}
&(a) \quad a=([1]^6,[2]^{g-6})\mbox{ or}\\
&(b) \quad a=([1]^4,[3]^2,[2]^{g-6}),
\end{split}\end{equation}

and the dimensions of the eigenspaces $V_1$, $V_2$, $V_3$ are the following:

$$(a) \quad \; \; \qquad \qquad d_1 = \frac{g+1}{2}; \qquad d_2 = 2; \qquad d_3 =   \frac{g-1}{2}-2;$$
$$(b) \qquad \qquad d_1 = \frac{g-1}{2} ; \qquad d_2 = 2; \qquad d_3 = \frac{g-1}{2}-1  .$$

These four distinct equivalence classes of data define the four distinct irreducible components of $\mathcal{BH}_{g,2,Gal}$. 

Since $g \geq 8$,  in each of the 4 cases, we have $d_1\geq 2$ and $d_3 \geq 2$, so we conclude by \cite{colombo2015totally}, 
Proposition 5.4.

 \end{proof}

\section{Rank of the second Gaussian map on the bielliptic locus}\label{sezione:rango_biell}

Here we will find a bound for the rank of the second Gaussian map on the general curve of the bielliptic locus ($g \geq 5$). Recall that the rank of the first Gaussian map of a bielliptic curve is known to be $3g-3$ (see~\cite[Section 3]{ciliberto1992gaussian}).

 As before, call $|L|$ and $|K-L|$ the adjoint $\mathfrak{g}_4^1$'s of a bielliptic curve $C$. Fix a basis for $H^0(L)= \langle x, y \rangle$, and a basis for $H^0(K-L) = \langle t_1, t_2, \dots, t_{g-3} \rangle$. Picking two independent sections $t_i, t_j \in H^0(K-L)$, it is possible to construct an invariant quadric in $I_2(K_C)$:
\begin{equation}\label{quadrica}
Q_{i,j} = x t_i \odot y t_j - x t_j \odot y t_i \in I_2(K_C) = I_2(K_C)^{\sigma}.
\end{equation}
From~\cite[Lemma 2.2]{colombo2008some}, the second Gaussian map decomposes as:
\begin{equation}
\mu _2 (Q_{i,j}) = \mu _{1,|L|}(x \wedge y) \mu _{1,|K-L|}(t_i \wedge t_j).
\end{equation}
Since $\mu _{1,|F|}(x \wedge y)$ is a fixed non-zero section in $H^0(K-2L)$, one has:
\begin{equation}\label{rk_mu2_mu1}
\rank \mu_2 \geq \rank \mu_2 |_{\langle Q_{i,j} \rangle} = \rank \mu _{1,|K-L|}.
\end{equation}
In order to find a lower bound for the rank of the second Gaussian map, we thus study the first Gaussian $\mu _{1,|K-L|} : \bigwedge^2 H^0(K-L) \rightarrow H^0(3K-2L)$.
We have the following diagram

\begin{tikzcd}
\bigwedge^2 H^0(K-L) \arrow[swap]{d}  \arrow{rr}{\mu _{1,|K-L|}} && H^0(3K-2L) \arrow{d} \\
\bigwedge^2 H^0(K) \arrow{rr}{\mu _{1,|K|}} & & H^0(3K)
\end{tikzcd} \\

where the vertical arrows are the inclusions. 


We will perform the computation on special bielliptic curves $C \ra E$ such that $[C]$ lies in the locus $\mathcal{B}_{g, Gal} := \mathcal{B}_{g,1, Gal}$.

With a similar computation as in Theorem \ref{bihygal} one can obtain all the possible monodromies of the cyclic coverings $\psi: C \ra \PP^1$ as in the diagram above and we get the following

\begin{prop}\label{componenti_biell}
For every $g \geq 3$ the locus $\mathcal{B}_{g,Gal}$ is non-empty. Moreover it is irreducible if $g$ is even and it has $2$ irreducible components in case $g$ odd. In particular, up to Hurwitz moves, we have the following possible monodromies:

\begin{enumerate}
\begin{multicols}{2}
\item If $g$ is odd, then:

$(a)$ $a=([1]^4,[2]^{g-3})$ or
 
$(b)$ $a=([1]^2,[3]^2,[2]^{g-3})$.

\item If $g$ is even, then:

$a=([1]^3,[3],[2]^{g-3})$.
\end{multicols}
\end{enumerate}
%
\end{prop}

In the following we compute the rank of the second Gaussian map, considering separately cases $g$ odd and $g$ even. The curves in the families of cyclic coverings as in the diagram are the normalisations of the curves whose equation is 
$y^4 = \prod_{i=1,..., g+1} (x-t_i)^{a_i}$ where for simplicity we choose $t_1= 0$. 

We can decompose  $H^0(K_C) = V_1 \oplus V_2 \oplus V_3$, where $V_j$ is the  eigenspace with eigenvalue $\zeta^j$, where $\zeta$ is a primitive fourth root of unity, e.g. $\zeta=i$. We have an explicit description of the $V_i$'s (see e.g.~\cite{moonen2010special}):
\begin{equation*}\begin{split}
&V_1 = \langle \omega_{1,0}, x \, \omega_{1,0}, \dots, x^{d_1-1} \, \omega_{1,0} \rangle;\\
&V_2 = \langle \omega _{2,0} \rangle;\\
&V_3 = \langle \omega_{3,0}, x \, \omega_{3,0}, \dots, x^{d_3-1} \, \omega_{3,0} \rangle,
\end{split}\end{equation*}
where 
\begin{equation}
\label{forms}
\omega_{n,0} = y^n \prod_{i=1,...,g+1}(x-t_i)^{l(i,n)} dx, \  \  l(i,n) =  \left \lfloor{\frac{-na_i}{4}}\right \rfloor, \ \ i =1,2,3.
\end{equation}

By a simple computation (see e.g.~\cite{moonen2010special}), if $g$ is odd the dimension of the eigenspaces in all irreducible components are: 
$$(a) \qquad \qquad d_1 = \frac{g+1}{2}; \qquad d_2 = 1; \qquad d_3 = \frac{g-3}{2};$$
$$(b) \qquad \qquad d_1 = \frac{g-1}{2}; \qquad d_2 = 1; \qquad d_3 = \frac{g-1}{2}.$$

If $g$ is even then 
$$\qquad \qquad d_1 = \frac{g}{2}; \qquad d_2 = 1; \qquad d_3 = \frac{g-2}{2}.$$

So one immediately gets a basis of $H^0(K-L)$ given by $$\{x \, \omega_{1,0}, \dots, x^{d_1-1} \, \omega_{1,0}, x \, \omega_{3,0}, \dots, x^{d_3-1} \, \omega_{3,0}\}.$$
Hence $H^0(K-L)$ decomposes as the direct sum of the $(\pm i)$- eigenspaces: 
\begin{equation*}\begin{split}
&W_1 = \langle x \, \omega_{1,0}, \dots, x^{d_1-1} \, \omega_{1,0} \rangle,\\
&W_3 = \langle x \, \omega_{3,0}, \dots, x^{d_3-1} \, \omega_{3,0} \rangle.
\end{split}\end{equation*}

So we have 
\begin{equation}\label{decomp}
\wedge ^2 H^0(K-L) = \wedge ^2 W_1 \oplus \wedge ^2 W_3 \oplus \Big( W_1 \otimes W_3 \Big).
\end{equation}

Clearly the group $G = \Z/4Z$ acts as a multiplication by $-1$ on the subspace $\wedge ^2 W_1 \oplus \wedge ^2 W_3$ and as the identity on  $W_1 \otimes W_3$. 

Since the first Gaussian map is $G$-equivariant, we can consider the map acting separately on eigenspaces relative to different eigenvalues. Call $Z_0$ and $Z_2$ the eigenspaces in $H^0(3K-2L)$ relative to the eigenvalues $1$ and $-1$ respectively. The following diagram holds:

\[
\raisebox{.9\baselineskip}{$_{\begin{array}{ccccc}
\mu_{1,|K-L|} & : & \wedge^2 H^0(K-L) & \longrightarrow & H^0(3K-2L) \\
{} & {} & \veq & {} & \cup \\
{} & {} & W_1 \otimes W_3 & \longrightarrow & Z_0 \\
{} & {} & \oplus & {} & \oplus \\
{} & {} & \wedge^2W_1 \oplus \wedge^2 W_3 & \longrightarrow & Z_2. \end{array} }$},
\]

In the following we will bound, separately, the rank of $\mu_{1,|K-L|}$ when restricted to $\wedge^2W_1 \oplus \wedge^2 W_3$ and to $W_1 \otimes W_3$. 

Notice that both in case $g$ odd and $g$ even we have $a_1=1$, hence the equation is 

\begin{equation}
\label{Equation}
y^4 = x \prod_{i=2}^{g+1} (x-t_i)^{a_i},
\end{equation} 
so $y$ is a local coordinate in a neighbourhood of the critical point $p$ where $y=x=0$. Thus we write $x = \phi(y)$ locally around  $p$. 

\begin{teo}\label{rango_wedge} In every irreducible component of $\mathcal{B}_{g,Gal}$ the following holds:
$$\rank(\mu_{1,|K-L|})\big|_{\wedge^2W_1}=2d_1-5; \qquad\rank(\mu_{1,|K-L|})\big|_{\wedge^2W_3}=2d_3-5.$$
\end{teo}
\begin{proof}
Since the procedure is the same, we focus on $W_1$. Choose a local chart and write $\omega_{1,0} = f(y)dy$. Denote $v_i = \phi(y)^i f(y) dy$, so that $W_1 =  \langle v_1, \dots, v_{d_1-1} \rangle$. Performing the computation for the first Gaussian map we get:
\begin{equation*}\begin{split}
\mu_{1,|K|}(v_i \wedge v_j) &= (i \phi(y)^{i-1} \phi'(y) f(y) + \phi(y) ^i f'(y)) \phi(y)^j f(y) - \\
&\qquad -\phi(y)^i f(y) (j \phi(y)^{j-1} \phi'(y) f(y) + \phi(y) ^j f'(y))(dy)^3=\\
&=(i-j)\phi(y)^{i+j-1}\phi'(y)f(y)^2(dy)^3.
\end{split}\end{equation*}

So, if we fix a $k \in \{3,...,2d_1-3\}$, for any $(i,j)$ such that $i<j$ and $i+j=k$, we have 
$$\mu_{1,|K|}(v_i \wedge v_j) = (i-j)\phi(y)^{k-1} \phi'(y) f(y)^2,$$ therefore the rank of the restriction of $\mu_{1, |K|}$ to $\Lambda^2 W_1$ is exactly $2d_1 - 5$. 

This implies that $\rank(\mu_{1,|K-L|})|_{\wedge^2W_1}= 2d_1-5$. In the same way one obtains $\rank(\mu_{1,|K-L|})|_{\wedge^2W_3}= 2d_3-5$.
\end{proof}

Since $\rank \mu_1 \big|_{\wedge^2W_1 \oplus \wedge^2 W_3} \geq \operatorname{max} \,(\rank \mu_1 \big|_{\wedge^2W_1}, \rank \mu_1 \big|_{\wedge^2W_3})$,  substituting the dimension of $d_i$'s depending on the chosen monodromy, if $g$ is odd we have:
\begin{equation}\begin{aligned}
&(a) \qquad  \rank(\mu_{1,|K-L|})|_{\wedge^2W_1 \oplus \wedge^2 W_3} \geq g-4; \\
&(b) \qquad  \rank(\mu_{1,|K-L|})|_{\wedge^2W_1\oplus \wedge^2 W_3}\geq g-6.
\end{aligned}\end{equation}
If $g$ is even we have 

\begin{equation}\begin{aligned}
&\qquad  \rank(\mu_{1,|K-L|})|_{\wedge^2W_1 \oplus \wedge^2 W_3}\geq g-5; \\
\end{aligned}\end{equation}

\begin{teo}\label{rango_prod} In every irreducible component of $\mathcal{B}_{g,Gal}$ the following holds:
\begin{equation}
\rank(\mu_{1,|K-L|})|_{W_1 \otimes W_3} \geq g-4.
\end{equation}
\end{teo}

\begin{proof}
As above choose the local coordinate $y$ around $0$. In local coordinates:
\begin{equation}
\omega_{1,0} = f(y)dy, \qquad \qquad \omega_{3,0} = g(y)dy.
\end{equation}

Write as above  the vector space $H^0(K-L)$ as a direct sum of:
\begin{equation}\begin{split}
W_1 = \langle x \, \omega_{1,0}, \dots, x^{d_1-1} \, \omega_{1,0} \rangle = \langle v_1 ,\dots, v_{d_1-1} \rangle, \\
W_3 = \langle x \, \omega_{3,0} ,\dots, x^{d_3-1} \, \omega_{3,0} \rangle = \langle w_1 ,\dots, w_{d_3-1} \rangle,
\end{split}\end{equation}
 where $v_i = \phi(y)^i f(y) dy$, and $w_j = \phi(y)^j g(y)dy$.

Using the equation \eqref{Equation} one immediately computes  
\begin{equation*}
f(y) = y \prod_{i=1,...,g+1} (\phi(y) - t_i)^{l(i,1)} \phi'(y) = \frac{\phi'(y)}{y^3}  \prod_{i=1,...,g+1} (\phi(y) - t_i)^{l(i,1) + a_i}
\end{equation*}

\begin{equation*}
g(y) = y^3 \prod_{i=1,...,g+1} (\phi(y) - t_i)^{l(i,3)} \phi'(y) = \frac{\phi'(y)}{y^3} y^2 \prod_{i=1,...,g+1} (\phi(y) - t_i)^{l(i,3) + a_i}.
\end{equation*}

Denote by $$\tilde{f} = \prod_{i=1,...,g+1} (\phi(y) - t_i)^{l(i,1) + a_i}$$ and by $$\tilde{g} = y^2 \prod_{i=1,...,g+1} (\phi(y) - t_i)^{l(i,3) + a_i}.$$

We compute the first Gaussian map on the wedge product $v_i \wedge w_j$, obtaining:
\begin{equation}
\label{gaussian1}
\mu_{1, |K|}(v_i \wedge w_j) = \frac{(\phi'(y))^2}{y^6}  ((i-j) \phi' \tilde{f} \tilde{g} + \phi(\tilde{f}' \tilde{g}-\tilde{f} \tilde{g}')) \phi^{i+j-1} (dy)^3.  
\end{equation}

One computes that for every possible monodromy one gets $l(i,3) = -a_i$ for all $i$, hence $\tilde{g} = y^2$, and $l(1,1) = -1 = -a_1$, hence $\tilde{f}(0) \neq 0$, since we have chosen $t_1=0$ and $t_i \neq 0 $ for all $i >1$. 

Consider Taylor expansions centered in $0$ for all factors appearing in equation \eqref{gaussian1}. 
Equation \eqref{Equation} can be written as $y^4 = x h(x) $, where $h(x)  = \prod_{i =2}^{g+1} (x-t_i)^{a_i}$. So one computes 
\begin{equation*}
\begin{split}
&\phi(y) =a y^4 +b y^8 + \mbox{ h.o.t.}, \ a = \frac{1}{h(0)}, \ b = - \frac{h'(0)}{(h(0))^2} \\
&\phi'(y) = 4a y^3 + 8by^7 + \mbox{ h.o.t.}\\
\end{split}
\end{equation*}

So the term of lowest degree in the Taylor expansion of \eqref{gaussian1} is 

$$32 a^{i+j+2} \tilde{ f}(0) (2(i-j) -1) y^{4(i+j-1)}. $$

So for every $k \in \{2,...,d_1+d_3-2\}$ and for a pair$(i,j)$ with $i+j= k$, $1 \leq i \leq d_1 -1$, $1\leq j \leq d_3 -1$, the order of vanishing of $\mu_{1, |K|} (v_i \wedge w_j)$ at the origin is $4(i+j-1)= 4k-4$.  

So, if for every $k \in \{2,...,d_1+d_3-2\}$  we choose exactly one pair $(i,j)$ with $i+j = k$, we find that the images of such elements under $\mu_{1,|K|}$ are all linearly independent since they have distinct vanishing order at the origin ($= 4k-4$). This shows that the rank of the restriction of the first gaussian map to $W_1 \otimes W_3$ is at least $d_1 + d_3 -3 = g-4$.

This proves theorem.
\end{proof}

Putting together results of Theorem \ref{rango_wedge} and Theorem \ref{rango_prod} and equation \eqref{rk_mu2_mu1} one gets a bound for the rank of the second Gaussian map depending on the monodromy. 

\begin{prop}\label{mu2_sommario1}
The rank of the second Gaussian map on every irreducible component of  the locus  $\mathcal{B}_{g,Gal}$  satisfies the following bounds, depending on genus and monodromy:

if $g$ is odd 
\begin{equation}\label{boundodd}\begin{split}
&(a): \quad  \rank \mu_2 \geq 2g-8,\\
&(b): \quad  \rank \mu_2 \geq  2g-10.
\end{split}\end{equation}
If $g$ is even 
\begin{equation}\label{boundodd}\begin{split}
\quad  \rank \mu_2 \geq 2g-9.\\
\end{split}\end{equation}
\end{prop}

Finally, we  have  the following. 
\begin{teo}\label{mu2_sommario}
For every bielliptic curve of genus $g$, we have 
\begin{equation*}
 \rank \mu_2 \leq 5g-5.
\end{equation*}
The rank of the second Gaussian map on the general curve of the bielliptic locus satisfies the following bounds:

(1) If $g$ is odd then:
\begin{equation*}
 \quad 2g-8 \leq \rank \mu_2;
\end{equation*}

(2) If $g$ is even then:
\begin{equation*}
 \quad 2g-9 \leq \rank \mu_2.
\end{equation*}
\end{teo}

\begin{proof}
The lower bounds for $\rank \mu_2$ for the generic bielliptic curve follow immediately from Proposition \ref{mu2_sommario1} by specialisation to curves in  $\mathcal{B}_{g,Gal}$. 

It remains only to show that for every bielliptic curve, $ \rank \mu_2 \leq 5g-5$. 

Denote by $\sigma$ the bielliptic involution. 
The second Gaussian  map is $\sigma$-equivariant. This implies that on the bielliptic locus, the second Gaussian map takes values on the $\sigma$-invariant part of $H^0(4K_C)$:
\begin{equation}
\mu_2 : I_2(K_C)^{\langle \sigma \rangle} \rightarrow H^0(4K_C)^{\langle \sigma \rangle}.
\end{equation}
Therefore, in order to bound the corank of $\mu_2$, we find a lower bound for the dimension $h^0(4K_C)^-$ of the $(-1)$-eigenspace $H^0(4K_C)^-$. 

First of all we consider those elements of $H^0(4K_C)^-$ obtained via the multiplication of an invariant $1$-form with a section of $H^0(3K_C)^-$, that is the image of the multiplication map
\begin{equation}
H^0(K_C)^{\langle \sigma \rangle} \otimes H^0(3K_C)^- \hookrightarrow H^0(4K_C)^-.
\end{equation}
Since $h^0(K_C)^{\langle \sigma \rangle}=1$, this map is injective. This shows that the dimension of $H^0(4K_C)^-$ is bounded below by the dimension of $H^0(3K_C)^-$. 
%

Now consider elements in $H^0(3K_C)^-$ obtained as a product between an anti-invariant $1$-form with a section of $H^0(2K_C)^{\langle \sigma \rangle}$. So fix an element $\eta \in H^0(K_C)^-$ and  consider the restriction of the multiplication map to
\begin{equation}
\langle \eta \rangle \otimes H^0(2K_C)^{\langle \sigma \rangle} \hookrightarrow H^0(3K_C)^-,
\end{equation}
which is clearly injective. Since $h^0(2K_C)^{\langle \sigma \rangle}= 2g-2$, that is the dimension of the bielliptic locus,  we obtain $$h^0(4K_C)^- \geq h^0(3K_C)^- \geq 2g-2.$$
This implies $\operatorname{corank}\mu_2 \geq 2g-2$, so $\rank \mu_2 \leq 5g-5$ as required.
\end{proof}

\section{Computations in low genus}\label{section:computation_biellittic}
In this section we will use the \verb|MAPLE| code reported in Appendix \ref{appendix_code} to compute the rank of the second Gaussian map on bielliptic curves of genus $5 \leq g \leq 30$ that are in $\mathcal{B}_{g,Gal}$. The code gives us a lower bound for the rank of the second Gaussian map $\mu_2: I_2(K_C) \ra H^0(C,4K_C)$. Computing the dimensions of the vector spaces involved, one gets:
\begin{equation}
\dim I_2(K_C)= \frac{(g-2)(g-3)}{2}, \qquad \qquad \dim H^0(4K_C)=7g-7.
\end{equation}
For dimensional reasons, the map can be injective for genus $g \leq 17$ only. Moreover, recall from Theorem \ref{mu2_sommario} that the second Gaussian map has corank at least $2g-2$ over the bielliptic locus. In particular, if $g \geq 18$ the map can not be surjective, and it can neither be injective when $14 \leq g \leq 17$. 

\begin{center}
\begin{longtable}{|c|c|c|c|c|}
\caption{The rank for the second Gaussian map of bielliptic curves that are $\Z/4\Z$ covers of $\PP^1$. The rank is computed using MAPLE; the maximal rank is given by the minimum between $\dim I_2(K)$ and $5g-5$.}\label{table_inj}\\
\hline
genus & monodromy & $\rank \mu_2$ & max $\rank \mu_2$ \\
\hline
$5$ & $[1^4:2^2]$ & $3$  & $3$\\
\hline
$6$ & $[1^3:3:2^3]$ & $6$  & $6$\\
\hline
$7$ & $[1^4:2^4]$ & $10$  & $10$\\
\hline
$8$ & $[1^3:3:2^5]$ & $15$  & $15$\\
\hline
$9$ & $[1^4:2^6]$ & $\geq 17$ & $21$\\
\hline
$10$ & $[1^3:3:2^7]$ & $\geq 19$  & $28$\\
\hline
$11$ & $[1^4:2^8]$ & $\geq 21$  & $36$\\
\hline
$12$ & $[1^3:3:2^9]$ & $\geq 23$  & $45$\\
\hline
$13$ & $[1^4:2^{10}]$ & $\geq 25$  & $55$\\
\hline
$14$ & $[1^3:3:2^{11}]$ & $\geq 27$  & $65$\\
\hline
$15$ & $[1^4:2^{12}]$ & $\geq 29$  & $70$\\
\hline
$16$ & $[1^3:3:2^{13}]$ & $\geq 31$  & $75$\\
\hline
$17$ & $[1^4:2^{14}]$ & $\geq 33$  & $80$\\
\hline
$18$ & $[1^3:3:2^{15}]$ & $\geq 35$  & $85$\\
\hline
$19$ & $[1^4:2^{16}]$ & $\geq 37$  & $90$\\
\hline
$20$ & $[1^3:3:2^{17}]$ & $\geq 39$  & $95$\\
\hline
$21$ & $[1^4:2^{18}]$ & $\geq 41$  & $100$\\
\hline
$22$ & $[1^3:3:2^{19}]$ & $\geq 43$  & $105$\\
\hline
$23$ & $[1^4:2^{20}]$ & $\geq 45$  & $110$\\
\hline
$24$ & $[1^3:3:2^{21}]$ & $\geq 47$  & $115$\\
\hline
$25$ & $[1^4:2^{22}]$ & $\geq 49$  & $120$\\
\hline
$26$ & $[1^{3}:3:2^{23}]$ & $\geq 51$  & $125$\\
\hline
$27$ & $[1^4:2^{24}]$ & $\geq 53$  & $130$\\
\hline
$28$ & $[1^3:3:2^{25}]$ & $\geq 55$  & $135$\\
\hline
$29$ & $[1^4:2^{26}]$ & $\geq 57$  & $140$\\
\hline
$30$ & $[1^3:3:2^{27}]$ & $\geq 59$  & $145$\\
\hline
\end{longtable}
\end{center}
In Table \ref{table_inj} we exhibited one example of bielliptic curve for every genus $5 \leq g \leq 30$, and we reported a lower bound for the rank of the second Gaussian map for each of them. Remark that, from Table \ref{table_inj}, we find $\rank \mu_2 \geq 2g-1$ for every $g \geq 8$, so the general bielliptic curve of genus $8 \leq g \leq 30$ has the same property. From the computations we get the following

\begin{teo}\label{teo:sommario_rkmu2_comp}
The second Gaussian map on the bielliptic locus is generically injective if $5 \leq g \leq 8$, moreover it cannot be surjective for $g \geq 14$. The general bielliptic curve of genus $8 \leq g \leq 30$ satisfies:
\begin{equation}
\rank \mu_2 \geq 2g-1.
\end{equation}
\end{teo}

\begin{remark}
Notice that for all the examples computed the rank of the second Gaussian map is exactly $2g-1$. So we expect that the rank of the second Gaussian map for a general bielliptic curve of genus $g \geq 8$ is exactly $2g-1$.
\end{remark}

\appendix

\section{MAPLE script}\label{appendix_code}
This appendix describes the \verb|MAPLE| code.
We explain here the simple strategy behind the code. The purpose is to provide a lower bound for the rank of the second Gaussian map when evaluated over a (generic) curve which is a cyclic Galois cover of $\PP^1$. We fix the \verb|genus| of the curve, the order \verb|m| of the Galois group, and the monodromy datum \verb|a=<a_1,..,a_N>|. Moreover we fix the branch points \verb|t=<t_1,...,t_N>| as well. We consider only monodromies such that $a_1=1$, and we always choose $t_1=0$.

For convenience of the reader, we recall that the fiber over a fixed point $t_i \in t$ is the normalization of the affine curve:
\begin{equation}\label{eq_norm_appen}
y^m = g(x):=\prod_{i=1}^N (x-t_i)^{a_i}.
\end{equation}
Call $\phi$ the local inverse of $g$ around $0$, that is:
\begin{equation}\label{eq_invers_appen}
\phi(y)= g^{-1}(y^m)=x.
\end{equation}

We write forms $\omega_{n,\nu} \in H^0(K)$ concretely using expression:
\begin{equation}\label{equazione_forme_appen}
\omega _{n,\nu} =\frac{m}{g'(\phi(y))} \, y^{n-1} \phi(y)^\nu \prod _{i=1}^N (\phi(y)-t_i)^{l(i,n) + a_i} \, dy.
\end{equation}


Every form in (\ref{equazione_forme_appen}) is explicitly computable locally around $0$. It is easy to compute the second Gaussian map just using its definition. We do it, and we consider their Taylor expansions truncated at some fixed precision. We put all coefficients of the Taylor expansions in the matrix \verb|MatM2|. Finally, computing the rank of this matrix, we obtain an approximate value for the rank of $\mu_2$. 

The precision in the approximations depends on the parameter \verb|prec|, which we set at the beginning. It determines at which order all Taylor series stop. The results we obtain are lower bounds for the rank of the second Gaussian map.

In the following, we include and comment the \verb|MAPLE| source in case of a curve of genus $5$ which covers $\mathbb{P}^1$ with Galois group $G = \mathbb{Z}/4\mathbb{Z}$ and monodromy data $a=(1,1,3,3,2,2)$ over the branch points $t=(0,1,-1,2,-2,3)$. 

\begin{verbatim}
restart;

Typesetting:-Settings(functionassign = false);
with(PolynomialTools); with(LinearAlgebra);

genus := 5; m := 4;
a := <1, 1, 3, 3, 2, 2>; t := <0, 1, -1, 2, -2, 3>;

r := Dimension(a);
l := Matrix(r, m-1);
d := Vector(m-1);
forma := Vector(genus);
prec := 150;

#Dimension of \Lambda^2 H^0(K)
L := (1/2)*genus*(genus-1); 

#Dimension of I_2(K)
N := (1/2)*genus*(genus+1)-3*(genus-1);

k := 1;
Max1 := 0; Max2 := 0; Max3 := 0;


g := x -> mul((x-t(i))^a(i), i = 1 .. r);
phi := solve(g(x) = y^m, x)[1];

phiTay := y -> taylor(phi, y = 0, prec);

eq0 := x = convert(phiTay(y), polynom);
eq1 := y^m = g(x);
\end{verbatim}

In the previous lines we have initialized all variables. We fixed the \verb|prec| parameter to $150$. 
We called \verb|g| the function defined in (\ref{eq_norm_appen}), and \verb|phi| is local inverse around $0$. \verb|phiTay| is the polynomial version of \verb|phi| truncated at order \verb|prec|. Equations \verb|eq1| and \verb|eq0| are respectively equation (\ref{eq_norm_appen}) and equation (\ref{eq_invers_appen}) in polynomial form.  We used the command \verb|convert(phiTay(y),polynom)| to get rid of the infinitesimal term $o(y^{\operatorname{prec}})$ in the Taylor series.

\begin{verbatim}
for n from 1 to m-1 do 
   for i from 1 to r do 
      l[i, n] := floor(-n*a(i)/m); 
      d[n] := -1+add(-n*a(j)/m-floor(-n*a(j)/m), j = 1 .. r);
   end do;
end do;

for n from 1 to m-1 do 
   for v from 0 to d[n]-1 do 
      wpar[n, v] := m*y^(n-1)*x^v*mul((x-t(i))^(l[i, n]+a(i)), 
         i = 1 .. r); 
      fnum[n, v] := convert(taylor(algsubs(eq0, wpar[n, v]), 
         y = 0, prec), polynom); 
      fden[n, v] := convert(taylor(algsubs(eq0, diff(g(x), x)), 
         y = 0, prec), polynom); 
      forma[k] := convert(taylor(fnum[n, v]/fden[n, v], y = 0, 
         prec), polynom); 
      k := k+1;
   end do;
end do;
\end{verbatim}

Here we computed the combinatorial data \verb|l[i,n]| and \verb|d[n]| to construct all forms in $H^0(K)$ using expression (\ref{equazione_forme_appen}).
Then we converted them in polynomials, using the Taylor expansion of numerator and denominator separately, and considering the Taylor expansion of the quotient truncated at level \verb|prec|. Finally, we put all forms in vector \verb|forma|. We order the forms as follows: \\
\begin{center}
 $\lbrace$ \verb|forma[1],...,forma[genus]|$\rbrace$
\end{center} 
\vspace{-1cm}
\begin{equation*}
\veq
\end{equation*}
\begin{equation*}
\lbrace \omega_{1,0}, \dots, \omega_{1,d_1-1}, \omega_{2,0}, \dots, \omega_{2,d_2-1}, \omega_{3,0}, \dots, \omega_{3,d_3-1}\rbrace.
\end{equation*}
\\

\begin{verbatim}
#Multiplication map and multiplication map of derivatives

k := 1; 
for i from 1 to genus do 
   for j from 1 to genus do 
      M[i, j] := convert(taylor(forma[i]*forma[j], y = 0, prec), 
         polynom); 
      CM[i, j] := CoefficientVector(M[i, j], y); 
      Max2 := max(Max2, Dimension(CM[i, j]));
      MD[i, j] := convert(taylor((diff(forma[i], y))*(diff(forma[j], 
         y)), y = 0, prec-1), polynom); 
      if i <= j then 
         M2[k] := MD[i, j]; 
         k := k+1;
      end if; 
   end do; 
end do;
\end{verbatim}

Here we constructed the matrix \verb|M|, having in each entry the Taylor series for \verb|forma[i]*forma[j]|, and then we isolated the coefficients in the multimatrix \verb|CM|. Finally we constructed the matrix \verb|MD|, such that the entry \verb|MD[i,j]| contains the Taylor series for \verb|forma'[i]*forma'[j]|. In the last \emph{if} cycle, we ordered vectors contained in the \verb|MD| in the simpler matrix \verb|M2|.

\begin{verbatim}
#Set the right length of coefficient vectors

for i from 1 to genus do 
   for j from 1 to genus do 
      for q from Dimension(CM[i, j])+1 to Max2 do
         CM[i, j](q) := 0;
      end do;
   end do;
end do;
\end{verbatim}

This is a technical \emph{for} cycle, useful to guarantee that all columns in \verb|C1| and \verb|CM| have the same length. To be more precise, until now \verb|C1| and \verb|CM| were not matrices, but vectors whose entry were other vectors of a-priori different length. Here we homogenize all lengths adding zeros when necessary.

\begin{verbatim}             
#Matrix of coefficients of the multiplication map

MatM := Vector(prec);
for i from 1 to genus do 
   for j from i to genus do 
     MatM := <MatM, CM[i, j]>;
   end do;
end do;
MatM := DeleteColumn(MatM, 1);

Rank(MatM);

                               12                     
\end{verbatim}

Here we use for the multiplication map the same strategy as before: we ordered the vectors contained in \verb|CM| in matrix \verb|MatM| (we will use this matrix later to describe quadrics in the $I_2(K)$). Finally, we computed \verb|Rank(MatM)|, which is a lower bound for the rank of the multiplication map, allowing us to check whether the curve is hyperelliptic or not.

\begin{verbatim}               
#Computation of the second Gaussian map:

K := NullSpace(MatM);
Max3 := 0; 
for k from 1 to N do 
   Omega[k] := add(K[k][i]*M2[i], i = 1 .. L+genus);
   F[k] := CoefficientVector(Omega[k], y);
   Max3 := max(Max3, Dimension(F[k]));
end do;

for k from 1 to N do 
   for q from Dimension(F[k])+1 to Max3 do
      F[k](q) := 0;
   end do;
end do;
\end{verbatim}

Here we computed the second Gaussian map starting from the matrix \verb|MatM|. The output of command \verb|NullSpace(MatM)| is a matrix whose columns are vectors in the kernel of \verb|MatM|. We call it \verb|K|. In the first \emph{for} cycle we computed the second Gaussian map and isolate the coefficients. In the last \emph{for} cycle we adjusted the dimensions adding zeros to make \verb|F| a matrix, as before. 

\begin{verbatim}
#Matrix of coefficients of the second Gaussian map

Mat := Vector(prec-1);
for i from 1 to N do 
   Mat := <Mat, F[i]>;
end do;
Mat := DeleteColumn(Mat, 1);

Rank(Mat);
                               3
\end{verbatim}

In the last few lines we computed the approximate rank of the second Gaussian map.  \verb|Rank(Mat)| is a lower bound for $\rank \mu_2$. Nevertheless, in this case the approximate values coincide with the maximal one: we can conclude that the map is injective.

\bibliographystyle{abbrv}

\end{document}